\documentclass{amsart}
\newtheorem{theorem}{Theorem}[section]
\newtheorem{lemma}[theorem]{Lemma}

\newtheorem{proposition}[theorem]{Proposition}

\newtheorem{example}[theorem]{Example}
\newtheorem{remark}[theorem]{Remark}
\def\erre{{\rm I\!R}}

\def\enne{{\rm I\!N}}

\def\RR{{\rm I\!R}}

\def\phi{\varphi}
\def\div{\mathop{\rm div}}

\title[Yamabe-type
equations]{Yamabe-type
equations on Carnot groups}
\author{Giovanni Molica Bisci}
\address[G. Molica Bisci]{Dipartimento P.A.U., Universit\`a  degli
Studi Mediterranea di Reggio Calabria, Salita Melissari - Feo di
Vito, 89100 Reggio Calabria, Italy} \email{gmolica@unirc.it}
\author{Du\v san Repov\v s}
\address[D. Repov\v s]{Faculty of Education
 and Faculty of Mathematics and Physics, University of Ljubljana,  1000 Ljubljana, Slovenia  
}
\email{dusan.repovs@guest.arnes.si}
\thanks{Typeset by \LaTeX}
\begin{document}

\begin{abstract}
This article is concerned with a class of elliptic equations on Carnot groups depending of one real positive parameter and involving a critical nonlinearity.
As a special case of our results we prove the existence of at least one nontrivial solution for a subelliptic critical equation defined on
a smooth and bounded domain $D$ of the {Heisenberg group}
$\mathbb{H}^n=\mathbb{C}^n\times \erre$.
Our approach is based on pure variational methods and locally sequentially weakly lower semicontinuous arguments.
\end{abstract}
\maketitle
%%%%%%%%%%%%%%%%%%%%%%%%%%%%%%%%%%%%%%%%%%%%%%%%%%%%%%%%%%%%%%%%%%%%%%%%%%%%%%%%%%%%%%%%%%%%%%%%%%%%%%%%

\section{Introduction}

In this paper we study the critical boundary value problem
$$(P_{\mu,\lambda}^{g})\,\,\,\,\,\,\,
\left\{
\begin{array}{ll}
-\Delta_{\mathbb{G}} u=\displaystyle \mu|u|^{2^*-2}u+\lambda g(u) &  \mbox{\rm in } D\\
u|_{\partial D}=0, &
\end{array}\right.
$$
where $\Delta_{\mathbb{G}}$ is a sublaplacian on a Carnot group ${\mathbb{G}}$, $2^*$ is the critical Sobolev exponent for $\Delta_{\mathbb{G}}$,
$D$ is a smooth bounded domain of $\mathbb{G}$, $\mu$ and $\lambda$ are real positive parameters, and $g$ denotes a subcritical continuous function.

Equations like $(P_{\mu,\lambda}^{g})$ naturally arise in the study of the Yamabe problem for a CR manifold $(M,g)$. This is the problem of finding a contact form
$\widetilde{g}$ on $M$ with prescribed constant Webster scalar curvature that, as is well-known, is equivalent to finding a positive function $w\in C^{\infty}(M)$ such that
$$(Y)\,\,\,\,\,\,\,\left\{
\begin{array}{ll}
\widetilde{g}=w^{4/(q-2)}g \\
-\Delta_{M}w+k_gw=k_{\widetilde{g}}w^{\frac{q+2}{q-2}},
\end{array}\right.
$$
where $k_g$ and $k_{\widetilde{g}}$ are respectively the scalar curvature of $(M,{g})$ and $(M,\widetilde{g})$, and $\Delta_M$ denotes the Laplace-Beltrami operator on $M$.\par
In particular, let $h_0$ be the standard contact form of the sphere $\mathbb{S}^{2n+1}$. Given a smooth function $\widetilde{\phi}$ on $\mathbb{S}^{2n+1}$, the Webster scalar curvature problem of $\mathbb{S}^{2n+1}$ consists of finding a constant form $h$ conformal to $h_0$ such that the corresponding Webster scalar curvature is $\widetilde{\phi}$.\par
 This problem is equivalent to solving the following equation
\begin{equation}\label{dim0}
c_n\Delta_{h}w(\sigma)+\frac{n(n+1)}{2}w(\sigma)=\widetilde{\phi}(\sigma)w(\sigma)^{\frac{c_n}{2}},\quad\sigma\in \mathbb{S}^{2n+1}
\end{equation}
where $\Delta_{h}$ is the Laplace-Beltrami operator on $(\mathbb{S}^{2n+1},h_0)$ and $c_n:=2(1+{1}/{n}).$\par
Using the Cayley map on the Heisenberg group $\mathbb{H}^n$,
equation \eqref{dim0}, up to a constant, assumes the following form
\begin{equation*}\label{dim00}
-\Delta_{\mathbb{H}^n} u=\phi(\xi) u^{\frac{q+2}{q-2}},\quad\xi\in\mathbb{H}^n
\end{equation*}
where $\Delta_{\mathbb{H}^n}$ is the {Kohn-Laplacian} operator on $\mathbb{H}^n$ and $q:=2n+2$ is the homogeneous dimension of $\mathbb{H}^n$ (see \cite{Lanco} for details).\par
 Hence the study of critical equations on stratified Lie groups is strictly connected to the above problem. However, the greatest part of the literature is devoted to the Heisenberg group $\mathbb{H}^n$, while only few results are known for the general Carnot
setting. For related topics, see, among others, the papers \cite{BK, BFP, GaLa,Lo1, MF, PiVa} and references therein.\par
For instance, in \cite[Theorems 1.1 and 1.2]{citti}, Citti studied the critical semilinear problem
$$(P_{f})\,\,\,\,\,\,\,
\left\{
\begin{array}{ll}
-\Delta_{\mathbb{H}^n} u+a(\xi)u=\displaystyle u^{\frac{q+2}{q-2}}+f(\xi,u) &  \mbox{\rm in } \Omega\\
u|_{\partial \Omega}=0, &
\end{array}\right.
$$
\noindent where $\Omega\subset \erre^{2n+1}$ is a smooth bounded domain, $a\in L^{\infty}(\Omega)$, $q$ is again the homogeneous dimension of $\mathbb{H}^n$, and $f$ denotes a suitable subcritical continuous function. In this case, the main ingredient of the proof was the explicit knowledge of the Sobolev minimizers for the Heisenberg gradient (see the paper of Jerison and Lee \cite{JLe}).\par
In general, by using suitable versions of the classical Concentration-Compactness principle established by Lions in \cite{Lions}, existence results for subelliptic problems can be found
by imposing conditions that permit a comparison of the critical levels for suitable Euler-Lagrange functionals.\par

For instance, Garofalo and Vassilev studied in \cite{GaLa0} the best constant in the Folland-Stein embedding on $\mathbb{G}$, by using
concentration-compactness arguments. Successively, exploiting this result and by means of a deep analysis developed by
Bonfiglioli and Uguzzoni in \cite{BoUg}, Loiudice studied in \cite{Lo1} the existence of positive and sign changing
solutions for critical problems on Carnot groups, extending to this subelliptic
context the results obtained by Br\'{e}zis and Nirenberg in
their pioneering paper \cite{BN}. See, among others, the papers \cite{AG,Cha, GP1, GP2, Zou} and references therein for related results.

It is natural to ask the question whether there are other approaches that give
rise to existence results for critical elliptic equations, and which do not require
the preceding comparison procedure.\par

Along this direction, in this paper (see Theorem \ref{FerraraMolicaBisci2}), by using direct variational methods, we prove
the existence of one weak solution in the Folland-Stein space $S^1_0(D)$ for problem $(P_{\mu,\lambda}^{g})$ without any use of
concentration-compactness techniques, provided that the parameter $\lambda$ is sufficiently small.\par
A key ingredient of the proof of Theorem \ref{FerraraMolicaBisci2} is a weakly lower semicontinuity trick (see Proposition \ref{ls}) previously used in literature for studying quasilinear $p$-Laplacian equations involving critical
nonlinearities in the Euclidean setting (see \cite{Faraci, Squa} and Remark \ref{FaraciFarkas}). More precisely, we prove that, for every $\mu>0$, the restriction of the functional
\begin{equation*}\label{PartFunzionale}
\mathcal{L}_{\mu}(u):=\frac{1}{2}\int_{D}|\nabla_{\mathbb{G}} u(\xi)|^2\,d\xi-\frac{\mu}{2^*}\displaystyle\int_D |u(\xi)|^{2^*}d\xi,\quad \forall\, u\in S^1_0(D)
\end{equation*}
\noindent to a sufficiently small ball in $S^1_0(D)$ is weakly lower semicontinuous.
 Consequently, the energy functional $$\mathcal{J}_{\mu,\lambda}(u):=\mathcal{L}_{\mu}(u)-\lambda\displaystyle\int_D \left(\int_0^{u(\xi)}g(\tau)d\tau\right) d\xi,\quad \forall\, u\in S^1_0(D)$$ associated to $(P_{\mu,\lambda}^{g})$, is locally
sequentially weakly lower semicontinuous, and by direct minimization we can prove
that for any $\mu>0$ and $\lambda$ sufficiently small, the functional $\mathcal{J}_{\mu,\lambda}$ admits a critical point (local minimum), which turns out to be a weak solution of problem $(P_{\mu,\lambda}^{g})$.\par
A special case of our result reads as follows.
\begin{theorem}\label{FerraraMolicaBisci22}
Let $D$ be a smooth and bounded domain of the {Heisenberg group}
$\mathbb{H}^n$ and
$g:\erre\rightarrow\erre$ a continuous function with $g(0)\neq 0$, for which
\begin{itemize}
\item[$(g_\infty')$] there exist $a_1, a_2>0$ and $p\in \left[1, \displaystyle 2\left(\frac{n+1}{n}\right)\right)$ such that
$$
|g(t)|\leq a_1+a_2|t|^{p-1},
$$
for every $t\in \erre$$.$
\end{itemize}
 Then there exists an open interval $\Lambda\subset (0,+\infty)$ such that for every $\lambda\in \Lambda$,
the following subelliptic problem
$$
\left\{
\begin{array}{ll}
-\Delta_{\mathbb{H}^n} u=\displaystyle |u|^{2^*_h-2}u+\lambda g(u) &  \mbox{\rm in } D\\
u|_{\partial D}=0, &
\end{array}\right.
$$
where $2^*_h$ is the critical Sobolev exponent for $\Delta_{\mathbb{H}^n}$,
admits at least one nontrivial weak solution in the {Folland-Stein space} $\mathbb{H}^1_0(D)$.
\end{theorem}
The interval $\Lambda$ in the above result can be explicitly localized. More precisely, setting
$$
c_{2^*_h}:=\sup_{u\in \mathbb{H}^1_0(D)\setminus\{0\}}\frac{\left(\displaystyle\int_D |u(\xi)|^{2^*_h}d\xi\right)^{1/2}}{\left(\displaystyle\int_{D}|\nabla_{\mathbb{H}^n} u(\xi)|^2\,d\xi\right)^{1/2}},
$$ one has
$$
\Lambda\subseteq \left(0,\max_{\varrho\geq 0}h(\varrho)\right),
$$
where
$$
h(\varrho):=
\frac{{\varrho}-c_{2^*_h}^{2^*_h}\varrho^{2^*_h-1}}{a_1c_{2^*_h}|D|^{\frac{2^*_h-1}{2^*_h}}+a_2c_{2^*_h}^{p}|D|^{\frac{2^*_h-p}{2^*_h}}\varrho^{p-1}},
$$
and $|D|$ denotes the Lebesgue measure of the smooth and bounded domain $D\subset \mathbb{H}^n$. See Remark \ref{intervalloo} for details.\par
Since the technique used in the paper does not require any Lie group structure the results that we state here are valid for more general operators than the sub-Laplacians on Carnot groups. The Carnot group structure is crucially used along the paper in the facts that the Sobolev constant $\kappa_{2^*}$ given in \eqref{ec} is finite and that the rational function $h_\mu$  introduced in Section \ref{Quattro}, formula \eqref{hfunc}, is well-defined (see also Remark \ref{FaraciFarkas2}).\par
\indent The plan of the paper is as follows. Section \ref{section2} is devoted to our abstract framework and basic definitions. Next, in Section \ref{Section3}, Proposition \ref{ls} and some preparatory results (see Lemmas \ref{lemmino1} and \ref{lemmino2}) are presented. In the last section, Theorem \ref{FerraraMolicaBisci2} is proved and some examples are proved (see Examples \ref{esempiuccio} and \ref{esempiuccio2}).

\section{Abstract Framework}\label{section2}

In this section we briefly recall some basic facts on Carnot groups and the functional space~$S^1_0(D)$. The special case of the Heisenberg group $\mathbb{H}^n$ is considered.
\subsection{Dilatations} Let $(\erre^n,\circ)$ be a Lie group equipped with a family of group automorphisms, namely \textit{dilatations}, $\mathfrak{F}:=\{\delta_\eta\}_{\eta>0}$ such that, for every $\eta>0$, the map
$$
\delta_\eta:\prod_{k=1}^{r}\erre^{n_k}\rightarrow\prod_{k=1}^{r}\erre^{n_k}
$$
is given by
$$
\delta_\eta(\xi^{(1)},...,\xi^{(r)}):=(\eta \xi^{(1)}, \eta^{2}\xi^{(2)},...,\eta^r\xi^{(r)}),
$$
where $\xi^{(k)}\in \erre^{n_k}$ for every $k\in \{1,...,r\}$ and $\displaystyle \sum_{k=1}^{r}n_k=n$.
\subsection{Homogeneous dimension} The structure $\mathbb{G}:=(\erre^n,\circ, \mathfrak{F})$ is called a \textit{homogeneous} group with \textit{homogeneous dimension}
\begin{equation}\label{dim}
\textrm{dim}_h{\mathbb{G}}:=\displaystyle \sum_{k=1}^{r}kn_k.
\end{equation}
\noindent From now on, we shall assume that $\textrm{dim}_h{\mathbb{G}}\geq 3$. We remark that, if $\textrm{dim}_h{\mathbb{G}}\leq 3$, then
necessarily $\mathbb{G}=(\erre^{\textrm{dim}_h{\mathbb{G}}},+)$.
Note that the number $\textrm{dim}_h{\mathbb{G}}$ is naturally associated to the family $\mathfrak{F}$ since, for every $\eta>0$, the Jacobian of the map
$$
\xi\mapsto \delta_\eta(\xi),\quad \forall\,\xi\in \erre^n
$$
equals $\eta^{\textrm{dim}_h{\mathbb{G}}}$.
\subsection{Stratification} Let $\mathfrak{g}$ be the Lie algebra of left invariant vector fields on $\mathbb{G}$ and assume that $\mathfrak{g}$ is \textit{stratified}, that is:
$$
\displaystyle\mathfrak{g}=\bigoplus_{k=1}^{r}V_k,
$$
where the integer $r$ is called the \textit{step} of $\mathbb{G}$, $V_k$ is a linear subspace of $\displaystyle\mathfrak{g}$, for every $k\in \{1,...,r\}$, and
\begin{itemize}
\item[] $\textrm{dim}V_k=n_k$, for every $k\in \{1,...,r\}$;
\item[] $[V_1,V_k]=V_{k+1}$, for $1\leq k\leq r-1$, and $[V_1,V_r]=\{0\}$.
\end{itemize}
In this setting the symbol $[V_1,V_k]$ denotes the subspace of $\mathfrak{g}$ generated by the commutators $[X,Y]$, where $X\in V_1$ and $Y\in V_k$.

\subsection{The notion of Carnot group and subelliptic Laplacian on $\mathbb{G}$} A\textit{Carnot group} is a homogeneous group $\mathbb{G}$ such that the Lie algebra $\mathfrak{g}$ associated to $\mathbb{G}$ is stratified.\par
Moreover, the \textit{subelliptic Laplacian} operator on $\mathbb{G}$ is the second-order differential operator, given by
$$
\Delta_{\mathbb{G}}:=\displaystyle\sum_{k=1}^{n_1}X_k^2,
$$
where $\{X_1,...,X_{n_1}\}$ is a basis of $V_1$. We shall denote by $$\nabla_\mathbb{G}:=(X_1,...,X_{n_1})$$ the
related \textit{horizontal gradient}.
\subsection{Critical Sobolev inequality} A crucial role in the functional analysis on Carnot groups is played by the
following Sobolev-type inequality
\begin{equation}\label{folland}
\int_{D}|u(\xi)|^{2^*}\,d\xi\leq C\int_{D}|\nabla_{\mathbb{G}} u(\xi)|^2\,d\xi,\,\quad\forall\, u\in C^{\infty}_0(D)
\end{equation}
due to Folland (see, for instance, \cite{Fo}). In the above expression $C$ is a positive constant (independent of $u$) and
$$
2^{*}:=\frac{2\textrm{dim}_h{\mathbb{G}}}{\textrm{dim}_h{\mathbb{G}}-2},
$$
is the \textit{critical Sobolev exponent}. Inequality (\ref{folland}) ensures that if $D$ is a bounded open (smooth) subset of $\mathbb{G}$, then the function
\begin{equation}\label{norm}
u\mapsto \|u\|_{S^1_0(D)}:=\left(\int_{D}|\nabla_{\mathbb{G}} u(\xi)|^2\,d\xi\right)^{1/2}
\end{equation}
is a norm in $C^{\infty}_0({D})$.
\subsection{Folland-Stein space} We shall denote by $S^1_0(D)$ the \textit{Folland-Stein space} defined
as the completion of $C^{\infty}_0({D})$ with respect to the norm $\|\cdot\|_{S^1_0(D)}$. The exponent $2^{*}$ is critical for $\Delta_{\mathbb{G}}$ since, as in the classical Laplacian setting, the embedding
$S^1_0(D)\hookrightarrow L^q(D)$ is compact when $1\leq q<2^{*}$, while it is only continuous if $q=2^{*}$, see Folland and Stein \cite{FoSte} and the survey paper \cite{Lanco} for related facts.
\subsection{The Heisenberg group} The simplest example of Carnot
group is provided by the \textit{Heisenberg group}
$\mathbb{H}^n:=(\erre^{2n+1},\circ)$, where, for every
$$
p:=(p_1,...,p_{2n},p_{2n+1})\,\,\,\mbox{and}\,\,\, q:=(q_1,...,q_{2n},q_{2n+1})\in \mathbb{H}^n,
$$
the usual group operation $\circ:\mathbb{H}^n\times \mathbb{H}^n\rightarrow \mathbb{H}^n$ is given by
$$
p\circ q:=\left(p_1+q_1,...,p_{2n}+q_{2n},p_{2n+1}+q_{2n+1}+\frac{1}{2}\displaystyle\sum_{k=1}^{2n}(p_kq_{k+n}-p_{k+n}q_k)\right)
$$
and the family of dilatations has the following form
$$
\delta_\eta(p):=(\eta p_1,...,\eta p_{2n},\eta^2p_{2n+1}),\quad \forall\, \eta>0.
$$
Thus $\mathbb{H}^n$ is a $(2n+1)$-dimensional group and by (\ref{dim}) it follows that
$$
\textrm{dim}_h{\mathbb{H}^n}=2n+2,
$$
and
$$
2_h^{*}:=\displaystyle 2\left(\frac{n+1}{n}\right).
$$
\noindent The Lie algebra of left invariant vector fields on $\mathbb{H}^{n}$ is denoted by $\mathfrak{h}$ and its standard basis
is given by \par
$$
X_k:=\partial_k-\frac{p_{n+k}}{2}\partial_{2n+1},\quad k\in \{1,...,n\}
$$
$$
Y_k:=\partial_{n+k}-\frac{p_{k}}{2}\partial_{2n+1},\quad k\in \{1,...,n\}
$$
$$
T:=\partial_{2n+1}.
$$

\noindent In this case, the only non-trivial commutators relations are
$$
[X_k,Y_k]=T,\quad \forall\, k\in \{1,...,n\}.
$$
Finally, the stratification of $\mathfrak{h}$ is given by
$$
\mathfrak{h}=\textrm{span}\{X_1,...,X_n,Y_1,...,Y_n\}\oplus \textrm{span}\{T\}.
$$
 We denote by $\mathbb{H}^1_0(D)$ the Folland-Stein space in the Heisenberg group setting, and by $\Delta_{\mathbb{H}^n}$ the {Kohn-Laplacian} operator on $\mathbb{H}^n$.\par
 The Heisenberg group plays a fundamental role in several curvature problems for
CR manifolds. Among the most important ones is, as recalled in Introduction, the CR-Yamabe problem, which
was completely solved by Jerison and Lee \cite{JLe}, and by Gamara \cite{gamara} and Gamara-Yacoub \cite{gamara2}.\par
\smallskip
\indent Finally, we cite the monograph \cite{BLU} for a nice introduction to Carnot groups and \cite{MRS} for related topics on variational methods used in this paper.

\section{Some technical Lemmas}\label{Section3}
Let $D$ be a smooth and bounded domain of a {Carnot group}
$\mathbb{G}$ with ${\rm dim}_h{\mathbb{G}}\geq 3$. From now on we shall denote by
\begin{equation}\label{ec}
\kappa_{2^*}:=\sup_{u\in S_0^1(D)\setminus\{0\}}\frac{\left(\displaystyle\int_D |u(\xi)|^{2^*}d\xi\right)^{1/2}}{\left(\displaystyle\int_{D}|\nabla_{\mathbb{G}} u(\xi)|^2\,d\xi\right)^{1/2}},
\end{equation}
the best constant of the continuous Sobolev embedding $S^1_0(D)\hookrightarrow L^{2^*}(D)$.\par
 By using the above notations we shall prove the following local weakly lower semicontinuity property that will be crucial in the sequel.
\begin{proposition}\label{ls}
Let $D$ be a smooth and bounded domain of a {Carnot group}
$\mathbb{G}$ with ${\rm dim}_h{\mathbb{G}}\geq 3$. Furthermore, let
$$\overline{B_{S^1_0(D)}(0,\varrho)}:=\left\{u\in S^1_0(D):\left(\displaystyle\int_{D}|\nabla_{\mathbb{G}} u(\xi)|^2\,d\xi\right)^{1/2}\leq \varrho\right\}$$
be the closed ball centered at $0$ and radius $\varrho>0$ in the {Folland-Stein space} $S^1_0(D)$. Then for every $\mu>0$, there exists $\varrho_{0,\mu}>0$ such that the functional
\begin{equation}\label{PartFunzionale}
\mathcal{L}_{\mu}(u):=\frac{1}{2}\int_{D}|\nabla_{\mathbb{G}} u(\xi)|^2\,d\xi-\frac{\mu}{2^*}\displaystyle\int_D |u(\xi)|^{2^*}d\xi
\end{equation}
is sequentially weakly lower semicontinuous on $\overline{B_{S^1_0(D)}(0,\varrho_{0,\mu})}$.
\end{proposition}

\begin{proof}
Let us fix $\mu>0$, let $\varrho>0$ and take a sequence $\{u_j\}_{j\in \enne}\subset \overline{B_{S^1_0(D)}(0,\varrho)}$ weakly convergent to some $u_\infty\in \overline{B_{S^1_0(D)}(0,\varrho)}$. This means that
\begin{equation}\label{convergenze000}
\begin{aligned}
 & \displaystyle\int_{D} \langle\nabla_\mathbb{G}u_j(\xi),\nabla_\mathbb{G}\varphi(\xi)\rangle\,d\xi \to \\
& \qquad \qquad \qquad
\displaystyle\int_{D} \langle\nabla_\mathbb{G}u_\infty(\xi),\nabla_\mathbb{G}\varphi(\xi)\rangle\,d\xi,
\end{aligned}
\end{equation}
for any $\phi\in S_0^1(D)$, as $j\to +\infty$. We will prove that
\begin{equation}\label{convergenzeFine}
L:=\liminf_{j\rightarrow +\infty}(\mathcal{L}_{\mu}(u_j)-\mathcal{L}_{\mu}(u_\infty))\geq 0,
\end{equation}
that is
$$
\begin{aligned}
 & \displaystyle \liminf_{j\rightarrow +\infty}\Bigg\{\frac{1}{2}\Bigg[\int_{D}|\nabla_{\mathbb{G}} u_j(\xi)|^2\,d\xi-\int_{D}|\nabla_{\mathbb{G}} u_\infty(\xi)|^2\,d\xi \Bigg] \\
& \qquad \qquad \qquad\qquad
-\frac{\mu}{2^*}\Bigg[\displaystyle\int_D |u_j(\xi)|^{2^*}d\xi-\displaystyle\int_D |u_\infty(\xi)|^{2^*}d\xi\Bigg]\Bigg\}\geq 0.
\end{aligned}
$$
Now, by using the following algebraic inequality (see \cite[Lemma 4.2]{Lin})
$$
|b|^2-|a|^2\geq 2\langle a,b-a\rangle+\frac{1}{2}|a-b|^2,\,\,\,\,(\forall\, a,b\in\erre^n),
$$
\noindent with the choices
$$
a:=\nabla_{\mathbb{G}} u_\infty(\xi),\,\,\,\,\,\,{\rm and }\,\,\,\,\,\, b:=\nabla_{\mathbb{G}} u_j(\xi),
$$
we get
\begin{equation}\label{GMB1}
\begin{aligned}
 & \frac{1}{2}\Bigg[\int_{D}|\nabla_{\mathbb{G}} u_j(\xi)|^2\,d\xi-\int_{D}|\nabla_{\mathbb{G}} u_\infty(\xi)|^2\,d\xi \Bigg]\\
& \qquad \qquad\qquad \geq \int_{D}\langle\nabla_\mathbb{G}u_\infty(\xi),\nabla_\mathbb{G}(u_j-u_\infty)(\xi)\rangle\,d\xi\\
& \qquad \qquad\qquad\qquad\qquad\qquad+ \frac{1}{4}\int_{D}|\nabla_{\mathbb{G}} (u_j-u_\infty)(\xi)|^2\,d\xi.
\end{aligned}
\end{equation}
On the other hand, by the well known Br\'{e}zis-Lieb Lemma (see \cite[Lemma 1.32]{Wi}), one has
\begin{equation}\label{GMB2}
\begin{aligned}
 & \liminf_{j\rightarrow +\infty}\Bigg[\displaystyle\int_D |u_j(\xi)|^{2^*}d\xi-\displaystyle\int_D |u_\infty(\xi)|^{2^*}d\xi\Bigg]\\
& \qquad \qquad\qquad = \liminf_{j\rightarrow +\infty}\displaystyle\int_D |u_j(\xi)-u_\infty(\xi)|^{2^*}d\xi.
\end{aligned}
\end{equation}
 Further, testing \eqref{convergenze000} with $\varphi=u_\infty$, it follows that
 \begin{equation}\label{convergenze0000}
\begin{aligned}
 & \displaystyle\int_{D} \langle\nabla_\mathbb{G}u_j(\xi),\nabla_\mathbb{G}u_\infty(\xi)\rangle\,d\xi \to \\
& \qquad \qquad \qquad
\displaystyle\int_{D} \langle\nabla_\mathbb{G}u_\infty(\xi),\nabla_\mathbb{G}u_\infty(\xi)\rangle\,d\xi,
\end{aligned}
\end{equation}
as $j\rightarrow +\infty$.\par
 \noindent Hence by using \eqref{GMB1} and \eqref{GMB2}, the above relation \eqref{convergenze0000} yields
 \smallskip
 \begin{equation}\label{perGMB22}
\begin{aligned}
 & \liminf_{j\rightarrow +\infty}(\mathcal{L}_{\mu}(u_j)-\mathcal{L}_{\mu}(u_\infty))\\
& \qquad \qquad\geq \liminf_{j\rightarrow +\infty}\Bigg\{\frac{1}{4}\int_{D}|\nabla_{\mathbb{G}} (u_j-u_\infty)(\xi)|^2\,d\xi\\
& \qquad \qquad\qquad\qquad\qquad\qquad-\frac{\mu}{2^*}\int_D |u_j(\xi)-u_\infty(\xi)|^{2^*}d\xi\Bigg\}.
\end{aligned}
\end{equation}
 Finally, bearing in mind that $S^1_0(D)\hookrightarrow L^{2^*}(D)$ continuously and owing to $\{u_j-u_\infty\}_{j\in \enne}\subset \overline{B_{S^1_0(D)}(0,2\varrho)}$, by \eqref{perGMB22} we easily have
\begin{equation}\label{GMB22}
\begin{aligned}
L &\geq \liminf_{j\rightarrow +\infty} \|u_j-u_\infty\|^{2}_{S^1_0(D)}\left(\frac{1}{4}-\frac{\mu\kappa_{2^*}^{2^*}}{2^*}\|u_j-u_\infty\|^{2^*-2}_{S^1_0(D)}\right)\\
& \geq \liminf_{j\rightarrow +\infty} \|u_j-u_\infty\|^{2}_{S^1_0(D)}\left(\frac{1}{4}-\frac{\mu\kappa_{2^*}^{2^*}2^{2^*-2}}{2^*}\varrho^{2^*-2}\right).
\end{aligned}
\end{equation}
Hence, for $\varrho$ sufficiently small, that is
$$
0<\varrho\leq \bar{\varrho}_\mu:=\frac{1}{2}\left(\frac{2^*}{4\mu \kappa_{2^*}^{2^*}}\right)^{\frac{1}{2^*-2}},
$$
by \eqref{GMB22} inequality \eqref{convergenzeFine} is verified.\par
 \noindent Therefore, the functional $\mathcal{L}_{\mu}$ is sequentially weakly lower semicontinuous on $\overline{B_{S^1_0(D)}(0,\varrho_{0,\mu})}$, provided that $\varrho_{0,\mu}\in (0,\bar{\varrho}_{\mu})$.
\end{proof}
\begin{remark}\rm{
Note that in \cite{Faraci, Squa}, studying elliptic problems
involving $p$-Laplacian operators on a bounded Euclidean smooth domain $\Omega$, the authors proved analogous weakly
lower semicontinuity results for the restriction of suitable energy functionals on small balls of $W^{1,p}_0(\Omega)$.}
\end{remark}
Fix $\mu,\lambda>0$ and denote
$$
\Phi(u):=\|u\|_{S^1_0(D)}
\quad\mbox{and}\quad\Psi_{\mu,\lambda}(u):=\frac{\mu}{2^*}\displaystyle\int_D |u(\xi)|^{2^*}d\xi+\lambda\int_{D}G(u(\xi))d\xi,
$$
for every $u\in S^1_0(D)$. Here, as usual, we set
\begin{equation}\label{G}
\displaystyle G(t):=\int_0^{t}g(\tau)d\tau,\quad\forall \,t\in\erre.
\end{equation}

The following preparatory results involving the functionals $\Phi$ and $\Psi_{\mu,\lambda}$ can be easily proved.

\begin{lemma}\label{lemmino1}
Let $\mu,\lambda>0$ and suppose that
\begin{equation*}\label{Ga199}
\limsup_{\varepsilon\rightarrow 0^+}\frac{\displaystyle \sup_{v\in\Phi^{-1}([0,\varrho_0])}\Psi_{\mu,\lambda}(v)-\sup_{v\in \Phi^{-1}([0,\varrho_0-\varepsilon])}\Psi_{\mu,\lambda}(v)}{\varepsilon}<\varrho_0,
\end{equation*}
for some $\varrho_0>0$. Then
\begin{equation}\label{Ga2}
\inf_{\sigma<\varrho_0}\frac{\displaystyle \sup_{v\in\Phi^{-1}([0,\varrho_0])}\Psi_{\mu,\lambda}(v)-\sup_{v\in  \Phi^{-1}([0,\sigma])}\Psi_{\mu,\lambda}(v)}{\varrho_0^2-\sigma^2}<\frac{1}{2}.
\end{equation}
\end{lemma}

Moreover, since the functional $\Psi_{\mu,\lambda}$ is sequentially weakly lower semicontinuous on $\Phi^{-1}([0,\varrho_0])$, the next proposition holds.

\begin{lemma}\label{lemmino2}
Let $\mu,\lambda>0$ and suppose that condition \eqref{Ga2} holds for some $\varrho_0>0$. Then
\begin{equation*}\label{VGa1}
\inf_{u\in\Phi^{-1}([0,\varrho_0))}
\frac{\displaystyle\sup_{v\in\Phi^{-1}([0,\varrho_0])}\Psi_{\mu,\lambda}(v)-\Psi_{\mu,\lambda}(u)}{\varrho_0^2- \|u\|_{S^1_0(D)}^2}<\frac{1}{2}.
\end{equation*}
\end{lemma}

See \cite{FMR} for more details.

\section{The Main Result}\label{Quattro}
The aim of this section is to prove that, for every $\mu>0$ and $\lambda$ sufficiently small, requiring that the perturbation term $g$ has a subcritical growth,
weak solutions to problem $(P_{\mu,\lambda}^{g})$ below do exist.\par

We recall that a \textit{weak solution} for the problem $(P_{\mu,\lambda}^{g})$, is a function $u:D\to \RR$ such that
$$
\mbox{
$\left\{\begin{array}{lll}
$$\displaystyle\int_{D} \langle\nabla_\mathbb{G}u(\xi),\nabla_\mathbb{G}\varphi(\xi)\rangle\,d\xi-\displaystyle\mu\displaystyle\int_D |u(\xi)|^{2^*-2}u(\xi)\varphi(\xi)d\xi\\
\qquad \qquad \qquad\qquad\quad\quad = \displaystyle\lambda\displaystyle\int_D g(u(\xi))\varphi(\xi)d\xi, \,\,\,\,\,\,\,\forall\,\varphi \in S^1_0(D)$$\\
$$u\in S^1_0(D)$$.
\end{array}
\right.$}
$$
\indent Our main result reads as follows.
\begin{theorem}\label{FerraraMolicaBisci2}
Let $D$ be a smooth and bounded domain of a Carnot group $\mathbb{G}$ with ${\rm dim}_h{\mathbb{G}}\geq 3$, and
$g:\erre\rightarrow\erre$ a continuous function such that
\begin{itemize}
\item[$(g_\infty)$] there exist $a_1, a_2>0$ and $p\in [1, 2^*)$ such that
$$
|g(t)|\leq a_1+a_2|t|^{p-1},
$$
for every $t\in \erre$.
\end{itemize}
 Then for every $\mu> 0$ there exists an open interval $\Lambda_\mu\subset (0,+\infty)$ such that for every $\lambda\in \Lambda_\mu$,
the following subelliptic critical problem
$$(P_{\mu,\lambda}^{g})\,\,\,\,\,\,\,
\left\{
\begin{array}{ll}
-\Delta_{\mathbb{G}} u=\displaystyle \mu|u|^{2^*-2}u+\lambda g(u) &  \mbox{\rm in } D\\
u|_{\partial D}=0, &
\end{array}\right.
$$
admits at least one weak solution in the Folland-Stein space $S^1_0(D)$.
\end{theorem}

\begin{proof}
Let us fix $\mu>0$ and let $\varrho_{\mu, \max}>0$ be the global maximum of the rational function defined by
\begin{equation}\label{hfunc}
h_\mu(\varrho):=
\frac{{\varrho}-\mu \kappa_{2^*}^{2^*}\varrho^{2^*-1}}{a_1\kappa_{2^*}|D|^{\frac{2^*-1}{2^*}}+a_2\kappa_{2^*}^{p}|D|^{\frac{2^*-p}{2^*}}\varrho^{p-1}},\quad\forall\varrho\geq 0.
\end{equation}
\indent Set $\varrho_{0,\mu}:=\min\{\varrho_{\mu, \max}, \bar{\varrho}_\mu\}$ and take
$$
\lambda\in\Lambda_\mu:=\left(0,h_\mu(\varrho_{0,\mu})\right).
$$
Hence there exists $\varrho_{0,\mu,\lambda}\in (0,\varrho_{0,\mu})$ such that
\begin{equation}\label{range}
\lambda<\frac{{\varrho_{0,\mu,\lambda}}-\mu \kappa_{2^*}^{2^*}\varrho_{0,\mu,\lambda}^{2^*-1}}{a_1\kappa_{2^*}|D|^{\frac{2^*-1}{2^*_h}}+a_2\kappa_{2^*}^{p}|D|^{\frac{2^*-p}{2^*}}\varrho_{0,\mu,\lambda}^{p-1}}.
\end{equation}
\indent Now, let us consider the functional $\mathcal{J}_{\mu, \lambda}:S^1_0(D)\to \RR$ defined by
$$
\mathcal{J}_{\mu, \lambda}(u):=\frac{1}{2}\|u\|_{S^1_0(D)}^2-\frac{\mu}{2^*}\displaystyle\int_D |u(\xi)|^{2^*}d\xi
$$
$$
\qquad\qquad\qquad\qquad\qquad\qquad-\lambda\displaystyle\int_D G(u(\xi))d\xi,\quad \forall\, u\in S^1_0(D)\,
$$
where $G$ is given in \eqref{G}.\par
Note that, since $g$ has a subcritical growth, the functional $\mathcal{J}_{\mu,\lambda}\in C^1(S^1_0(D))$ and its derivative at $u\in S^1_0(D)$ is given by
$$
\langle \mathcal{J}'_{\mu,\lambda}(u), \varphi\rangle = \displaystyle\int_{D} \langle\nabla_\mathbb{G}u(\xi),\nabla_\mathbb{G}\varphi(\xi)\rangle\,d\xi-\displaystyle\mu\displaystyle\int_D |u(\xi)|^{2^*-2}u(\xi)\varphi(\xi)d\xi
$$
$$
\qquad\qquad-\lambda\displaystyle\int_D g(u(\xi))\varphi(\xi)d\xi,
$$
for every $\varphi \in S^1_0(D)$. Thus the weak solutions of problem $(P_{\mu,\lambda}^{g})$ are exactly the critical points of the energy functional $\mathcal{J}_{\mu,\lambda}$. Now, let $0<\varepsilon<\varrho_{0,\mu,\lambda}$ and set
  \begin{equation*}\label{Ga1}
\Lambda_{\lambda,\mu}(\varepsilon,\varrho_{0,\mu,\lambda}):=\frac{\displaystyle \sup_{v\in\Phi^{-1}([0,\varrho_{0,\mu,\lambda}])}\Psi_{\mu,\lambda}(v)-\sup_{v\in \Phi^{-1}([0,\varrho_{0,\mu,\lambda}-\varepsilon])}\Psi_{\mu,\lambda}(v)}{\varepsilon}.
\end{equation*}
\indent Let us prove that
\begin{equation}\label{limsup}
\limsup_{\varepsilon\rightarrow 0^+}\Lambda_{\lambda,\mu}(\varepsilon,\varrho_{0,\mu,\lambda})<\varrho_{0,\mu,\lambda}.
\end{equation}
\noindent First, note that
 $$
 \Lambda_{\lambda,\mu}(\varepsilon,\varrho_{0,\mu,\lambda})\leq\frac{1}{\varepsilon}\left|\displaystyle \sup_{v\in\Phi^{-1}([0,\varrho_{0,\mu,\lambda}])}\Psi_{\mu,\lambda}(v)-\sup_{v\in \Phi^{-1}([0,\varrho_{0,\mu,\lambda}-\varepsilon])}\Psi_{\mu,\lambda}(v)\right|.
 $$
 Moreover, it follows that
 $$
 \Lambda_{\lambda,\mu}(\varepsilon,\varrho_{0,\mu,\lambda})\leq\sup_{v\in \Phi^{-1}([0,1])}\int_D\left|\int_{(\varrho_{0,\mu,\lambda}-\varepsilon)v(\xi)}^{\varrho_{0,\mu,\lambda} v(\xi)}\frac{|f_{\mu,\lambda}(t)|}{\varepsilon}dt\right|d\xi,
 $$
 where
 $$
 f_{\mu,\lambda}(t):=\mu|t|^{2^*-2}t+\lambda g(t),\quad\forall\,t\in\erre.
 $$
 On the other hand, the growth condition $(g_\infty)$ yields

$$
\sup_{v\in \Phi^{-1}([0,1])}\int_D\left|\int_{(\varrho_{0,\mu,\lambda}-\varepsilon)v(\xi)}^{\varrho_{0,\mu,\lambda} v(\xi)}\frac{|f_{\mu,\lambda}(t)|}{\varepsilon}dt\right|d\xi \leq \frac{\mu\kappa_{2^*}^{2^*}}{2}\left(\frac{\varrho_{0,\mu,\lambda}^{2^*}-(\varrho_{0,\mu,\lambda}-\varepsilon)^{2^*}}{\varepsilon}\right)
$$
$$
\qquad+\lambda\left(a_1\kappa_{2^*}|D|^{\frac{2^*-1}{2^*}}+a_2\frac{\kappa_{2^*}^{p}}{p}\left(\frac{\varrho_{0,\mu,\lambda}^{p}-(\varrho_{0,\mu,\lambda}-\varepsilon)^{p}}{\varepsilon}\right)|D|^{\frac{2^*-p}{2^*}}\right).
$$
 \noindent Thanks to the above relations, passing to the limsup, as $\varepsilon\rightarrow 0^+$, we get
 $$
\limsup_{\varepsilon\rightarrow 0^+}\Lambda_{\lambda,\mu}(\varepsilon,\varrho_{0,\mu,\lambda})<p_{\mu,\lambda}(\varrho_{0,\mu,\lambda}),
$$
where
$$
p_{\mu,\lambda}(\varrho_{0,\mu,\lambda}):=
\mu\kappa_{2^*}^{2^*}\varrho_{0,\mu,\lambda}^{2^*-1}+\lambda\left(a_1\kappa_{2^*}|D|^{\frac{2^*-1}{2^*}}+a_2\kappa_{2^*}^{p}|D|^{\frac{2^*-p}{2^*}}\varrho_{0,\mu,\lambda}^{p-1}\right).
$$
\noindent Bearing in mind \eqref{range}, it follows that
$$p_{\mu,\lambda}(\varrho_{0,\mu,\lambda})<\varrho_{0,\mu,\lambda}$$
and inequality \eqref{limsup} is verified.\par
\indent Owing to \eqref{limsup}, by Lemmas \ref{lemmino1} and \ref{lemmino2}, one has
 \begin{equation*}\label{VGa1Glu}
\inf_{u\in\Phi^{-1}([0,\varrho_{0,\mu,\lambda}))}
\frac{\displaystyle\sup_{v\in\Phi^{-1}([0,\varrho_{0,\mu,\lambda}])}\Psi_{\mu,\lambda}(v)-\Psi_{\mu,\lambda}(u)}{\varrho_{0,\mu,\lambda}^2- \|u\|_{S^1_0(D)}^2}<\frac{1}{2}.
\end{equation*}
The above relation implies that there exists $w_{\mu,\lambda}\in S^1_0(D)$ such that
$$
\Psi_{\mu,\lambda}(u)\leq \displaystyle\sup_{v\in\Phi^{-1}([0,\varrho_{0,\mu,\lambda}])}\Psi_{\mu,\lambda}(v)<\Psi_{\mu,\lambda}(w_{\mu,\lambda})+\frac{1}{2}(\varrho_{0,\mu,\lambda}^2-\|w_{\mu,\lambda}\|^2_{S^1_0(D)}),
$$
for every $u\in\Phi^{-1}([0,\varrho_{0,\mu,\lambda}])$. Thus
 \begin{equation}\label{John}
 \mathcal{J}_{\mu,\lambda}(w_{\mu,\lambda}):=\frac{1}{2}\|w_{\mu,\lambda}\|_{S^1_0(D)}^2-\Psi_{\mu,\lambda}(w_{\mu,\lambda})<\frac{\varrho_{0,\mu,\lambda}^2}{2}-\Psi_\lambda(u),
 \end{equation}
 for every $u\in\Phi^{-1}([0,\varrho_{0,\mu,\lambda}])$.\par
 \indent Since $\varrho_{0,\mu,\lambda}<\bar{\varrho}_\mu$, by Lemma \ref{ls} the energy functional $\mathcal{J}_{\mu,\lambda}$ is sequentially weakly lower semicontinuous on $\Phi^{-1}([0,\varrho_{0,\mu,\lambda}])$, the restriction $\mathcal{J}_{\mu,\lambda}|_{\Phi^{-1}([0,\varrho_{0,\mu,\lambda}])}$ has a global
 minimum $u_{0,\mu,\lambda}\in \Phi^{-1}([0,\varrho_{0,\mu,\lambda}])$.\par
  \indent Note that $u_{0,\mu,\lambda}$ belongs to $\Phi^{-1}([0,\varrho_{0,\mu,\lambda}))$. Indeed, if
 $\|u_{0,\mu,\lambda}\|_{S^1_0(D)}=\varrho_{0,\mu,\lambda}$, by \eqref{John}, one has
  \begin{equation*}
 \mathcal{J}_{\mu,\lambda}(u_{0,\mu,\lambda})=\frac{\varrho^2_{0,\mu,\lambda}}{2}-\Psi_{\mu,\lambda}(u_{0,\mu,\lambda})>\mathcal{J}_{\mu,\lambda}(w_{\mu,\lambda}),
 \end{equation*}
 \noindent which is a contradiction. In conclusion, it follows that $u_{0,\mu,\lambda}\in S^1_0(D)$ is a local minimum for the energy functional $\mathcal{J}_{\mu,\lambda}$ with $$\|u_{0,\mu,\lambda}\|_{S^1_0(D)}<\varrho_{0,\mu,\lambda},$$ hence in particular, a weak solution of problem $(P_{\mu,\lambda}^{g})$.
 The proof is now complete.
\end{proof}
\begin{remark}\label{intervalloo}\rm{
The interval $\Lambda_\mu$ in Theorem \ref{FerraraMolicaBisci2} can be explicitly localized. More precisely, one has
$$
\Lambda_\mu\subseteq \left(0,\max_{\varrho\geq 0}h_\mu(\varrho)\right),
$$
where
$$
h_\mu(\varrho):=
\frac{{\varrho}-\mu \kappa_{2^*}^{2^*}\varrho^{2^*-1}}{a_1\kappa_{2^*}|D|^{\frac{2^*-1}{2^*}}+a_2\kappa_{2^*}^{p}|D|^{\frac{2^*-p}{2^*}}\varrho^{p-1}},\quad(\forall\, \varrho\geq 0)
$$
and $|D|$ denotes the Lebesgue measure of the smooth and bounded domain $D\subset \mathbb{G}$. Notice that in the special case of the Heisenberg
group $\mathbb{H}^n$, an explicit expression of the Sobolev constant $\kappa_{2^*}$ was determined by Jerison and Lee in \cite[Corollary C]{JLe}.}
\end{remark}
\begin{remark}\label{FaraciFarkas2}\rm{
We observe that Theorem \ref{FerraraMolicaBisci2} remains valid also for subelliptic problems involving a more general subcritical lower order term, such as
$$(\widetilde{P}_{\mu,\lambda}^{g})\,\,\,\,\,\,\,
\left\{
\begin{array}{ll}
-\Delta_{\mathbb{G}} u=\displaystyle \mu|u|^{2^*-2}u+\lambda g(\xi,u) &  \mbox{\rm in } D\\
u|_{\partial D}=0, &
\end{array}\right.
$$
where $g:D\times\erre\rightarrow \erre$ is a non-zero Carath\'{e}odory function such that
$$
|g(\xi,t)|\leq a_1+a_2|t|^{r-1}+a_3|t|^{q-1},\quad \forall\, t\in\erre\,\, \mbox{and\, a.e.}\,\,\xi\in D,
$$
for some positive constants $a_j$ (with $j=1,2,3$) and $1<r<2\leq q<2^*$. In this case, for every $\mu> 0$, there exists an open interval $\Sigma_\mu\subset (0,+\infty)$ such that, for every $\lambda\in \Sigma_\mu$, problem $(\widetilde{P}_{\mu,\lambda}^{g})$ admits at least one weak solution in the Folland-Stein space $S^1_0(D)$. Furthermore the interval $\Sigma_\mu$ can be explicitly localized and one has
$$
\Sigma_\mu\subseteq \left(0,\max_{\varrho\geq 0}H_\mu(\varrho)\right),
$$
where
$$
H_\mu(\varrho):=
\frac{{\varrho}-\mu \kappa_{2^*}^{2^*}\varrho^{2^*-1}}{a_1\kappa_{2^*}|D|^{\frac{2^*-1}{2^*}}+a_2\kappa_{2^*}^{r}|D|^{\frac{2^*-r}{2^*}}\varrho^{r-1}+a_3\kappa_{2^*}^{q}|D|^{\frac{2^*-q}{2^*}}\varrho^{q-1}},
$$
for every non-negative $\varrho$.}
\end{remark}
\begin{remark}\label{FaraciFarkas}\rm{
 We observe that our variational methods are close to the ones adopted in \cite{Faraci} where the authors require suitable algebraic inequalities, involving the growth of the perturbation term $g$, and obtain some existence results for suitable classes of critical problems in bounded domains of the Euclidean space.
However, due to the presence of the real parameter $\lambda$, our technical approaches, as well as the statements of Theorems \ref{FerraraMolicaBisci2} and \ref{FerraraMolicaBisci22}, are different with respect to \cite{Faraci}.
For the sake of completeness, see also the papers \cite{AC,FMR} where a similar variational method has been used for studying subcritical elliptic equations.}
\end{remark}
\begin{remark}\label{Notzero}\rm{
Notice that if the identically zero function is not a local minimum for the energy functional $\mathcal{J}_{\mu, \lambda}$ then any weak solution of problem $(P_{\mu,\lambda}^{g})$ obtained by using Theorem \ref{FerraraMolicaBisci2} is non-trivial. This fact clearly happens, for instance, if $g(0)\neq 0$, as in Example \ref{esempiuccio} below.}
\end{remark}

\indent A direct application of our main result reads as follows.

\begin{example}\label{esempiuccio}\rm{
Let $D$ be a smooth and bounded domain of a Carnot group $\mathbb{G}$ with ${\rm dim}_h{\mathbb{G}}\geq 3$. By virtue of Theorem \ref{FerraraMolicaBisci2}, for every $\mu>0$ there exists an open interval $\Lambda_\mu\subset (0,+\infty)$ such that for every $\lambda\in \Lambda_\mu$,
the following critical problem
$$
\left\{
\begin{array}{ll}
-\Delta_{\mathbb{G}} u=\mu|u|^{2^*-2}u+\lambda\displaystyle (1+|u|^{p-1}) &  \mbox{\rm in } D\\
u|_{\partial D}=0, &
\end{array}\right.
$$
where $p\in [1,2^*)$,
admits at least one non-trivial weak solution in $S^1_0(D)$.
}
\end{example}
In the next example, bearing in mind Remark \ref{FaraciFarkas2}, the existence of one non-trivial weak solution is achieved in the case $g(0)=0$.
\begin{example}\label{esempiuccio2}\rm{
Let $D$ be a smooth and bounded domain of the Heisenberg group $\mathbb{H}^n$ with unitary Lebesgue measure. Furthermore, let $\alpha,\beta\in L^{\infty}_+(D)$ and take
$$1<r<2<q<2^*_h.$$ By virtue of Theorem \ref{FerraraMolicaBisci2} and Remark \ref{FaraciFarkas2}, for every $\mu>0$ there exists an open interval $\Sigma_\mu\subset (0,+\infty)$ such that for every $\lambda\in \Sigma_\mu$,
the following critical problem
$$
\left\{
\begin{array}{ll}
-\Delta_{\mathbb{H}^n} u=\mu|u|^{2^*_h-2}u+\lambda\displaystyle (\alpha(x)|u|^{r-2}u+\beta(x)|u|^{q-1}) &  \mbox{\rm in } D\\
u|_{\partial D}=0, &
\end{array}\right.
$$
has at least weak solution in $S^1_0(D)$.\par
 Notice that, in such a case, the identically zero function is not a local minimum for the energy functional $\mathcal{J}_{\mu, \lambda}$. Indeed, fix $u_0\in S^1_0(D)\setminus\{0\}$ and take $\zeta>0$. Then, as it is easy to see, there exists $\zeta_0>0$ such that
$$
\mathcal{J}_{\mu, \lambda}(\zeta u_0)<0,
$$
for every $\zeta\leq \zeta_0$.\par
 We conclude that the obtained weak solution is non-trivial and
$$
\Sigma_\mu\subseteq \left(0,\frac{1}{\min\{\kappa_{2^*_h},\kappa_{2^*_h}^{r}\|\alpha\|_{\infty},\kappa_{2^*_h}^{q}\|\beta\|_{\infty}\}}\max_{\varrho\in \Upsilon}\left(\frac{\varrho-\mu \kappa_{2^*_h}^{2^*_h}\varrho^{2^*_h-1}}{1+\varrho^{r-1}+\varrho^{q-1}}\right)\right),
$$
where
$$
\Upsilon:=\left(0,\frac{\mu^{\frac{1}{2-2^*_h}} }{\kappa_{2^*_h}^{\frac{2^*_h}{2^*_h-2}}}\right].
$$
}
\end{example}

\indent {\bf Acknowledgements.} The authors warmly thank the anonymous referee for her/his useful and nice
comments on the paper. The manuscript was realized within the auspices of the INdAM - GNAMPA Project 2016
titled {\it Problemi variazionali su variet\`a Riemanniane e gruppi di Carnot} and the SRA grant P1-0292, J1-5435 and J1-6721.

\end{document}